\documentclass[11pt,english]{article}
\usepackage[margin= 2 cm,bottom=23mm, top= 20mm]{geometry}

\usepackage{amsthm}
\usepackage{amsmath}
\usepackage{amssymb}
\usepackage{setspace}
\usepackage{mathtools}
\usepackage{verbatim}
\usepackage{csquotes}
\usepackage[hidelinks]{hyperref}
\usepackage{thm-restate}
\usepackage{cleveref}
\usepackage{enumitem}
\setlist[itemize]{leftmargin=*} 
\setlist[enumerate]{leftmargin=*}
\usepackage{framed}
\usepackage{floatrow}
\usepackage[T1]{fontenc}
\usepackage{bbm}
\usepackage[color=Orange!50!white, textwidth=22mm]{todonotes}

\usepackage{mathdots}
\usepackage{xcolor}
\usepackage{diagbox}
\usepackage{colortbl}

\usepackage{graphicx}
\usepackage{subcaption}

\theoremstyle{plain}

\newtheorem{theorem}{Theorem}
\newtheorem{claim}[theorem]{Claim}
\newtheorem{proposition}[theorem]{Proposition}

\theoremstyle{definition}

\newtheorem*{defn*}{Definition}

\newcommand{\calT}{\mathcal{T}}
\newcommand{\mc}{\mathrm{mc}}

\author{Domagoj Brada\v{c}\thanks{Institute of Mathematics, EPFL, Lausanne, Switzerland. Email: \textbf{domagoj.bradac@epfl.ch} This work was done while the author was at the Department of Mathematics, ETH Z\"urich, Z\"urich, Switzerland. Research supported in part by SNSF grant 200021-228014.}}

\title{On a question of Erd\H{o}s and Ne\v{s}et\v{r}il about minimal cuts in a graph}
\date{}

\begin{document}
    \maketitle

    \begin{abstract}
        Answering a question of Erd\H{o}s and Ne\v{s}et\v{r}il, we show that the maximum number of inclusion-wise minimal vertex cuts in a graph on $n$ vertices is at most $1.8899^n$ for large enough $n$.
        
    \end{abstract}
    
    Let $c(n)$ denote the maximum possible number of inclusion-wise minimal vertex cuts in a graph on $n$ vertices. Erd\H{o}s and Ne\v{s}et\v{r}il~\cite{erdos} asked to estimate $c(n)$. Denoting $\alpha \coloneqq \lim_n c(n)^{1/n}$ (we shall show the limit exists), they asked whether $\alpha < 2$. This problem also appears as problem \#150 on Thomas Bloom's Erd\H{o}s problems website~\cite{bloom-website}. We answer this question in the affirmative. More precisely, we prove the following.

    \begin{theorem} \label{thm:main}
        The maximum possible number of inclusion-wise minimal vertex cuts in a graph on $n$ vertices is at most $2^{(1+o(1)) H(1/3) n},$ where $H(x)$ is the binary entropy function. 
    \end{theorem}
    In other words, $\alpha \le 2^{H(1/3)} < 1.8899$. 

    On the other hand, Erd\H{o}s communicated the following construction of Seymour~\cite{erdos} which shows that $\alpha \ge 3^{1/3} > 1.4422.$ Let $G$ be a graph on $3m + 2$ vertices with two distinguished vertices connected by $m$ internally vertex-disjoint paths of length $4$. We can obtain $3^m$ minimal vertex cuts by taking exactly one internal vertex from each of these paths. So $\alpha \ge \lim_{m} 3^{m/(3m+2)} = 3^{1/3}.$

    After the publication of this article, the author was informed by Hans Raj Tiwary that the present results had already been known prior to this work. Fomin, Kratsch, Todinca and Villanger~\cite{fomin-kratsch} first proved that $\alpha < 2$, in fact, they showed $\alpha \le 1.7087$. The current best known bounds on $\alpha$ are
    \[ 1.4457 \le \alpha \le \frac{1 + \sqrt{5}}{2} \approx 1.618, \]
    where the upper bound was proved by Fomin and Villanger~\cite{fomin-villanger} and the lower bound by Gaspers and Mackenzie~\cite{gaspers-mackenzie} who also gave a simpler proof of the upper bound in~\cite{fomin-villanger}.
    
    It will be more convenient to work with a closely related function. Given a graph $G$ with two distinct vertices $u$ and $v,$ let $\mc_{u,v}(G)$ denote the number of inclusion-wise minimal sets separating $u$ and $v$ in $G$. More formally, $\mc_{u,v}(G)$ is the number of sets $T \subseteq V(G) \setminus \{u, v\}$ such that $u$ and $v$ are in different connected components of $G \setminus T,$ but $u$ and $v$ are in the same connected component of $G \setminus T'$ for any $T' \subsetneq T.$ Let $g(n)$ denote the maximum value of $\mc_{u,v}(G)$ across all choices of a graph $G$ on $n+2$ vertices and a pair of distinct vertices $u,v \in V(G)$.
    
    Clearly $g(n-2) \le c(n) \le \binom{n}{2} g(n-2),$ so $\lim_n c(n)^{1/n}$ exists if and only if $\lim_n g(n)^{1/n}$ exists in which case they are equal. With this notation, Seymour's construction actually shows that $g(3m) \ge 3^{m}$.   

    For completeness, we show that $\alpha$ is well-defined.
    \begin{proposition} \label{prop:limit-exists}
        The limit $\lim_n g(n)^{1/n}$ exists.
    \end{proposition}
    \begin{proof}
        Clearly $g(n) \le 2^n$ so the sequence $g(n)^{1/n}$ is bounded above by $2$. 
        
        We claim that for any $n, m \ge 1,$ it holds that $g(n+m) \ge g(n) g(m)$. Indeed, let $G_1$ be a graph on $n+2$ vertices with $g(n)$ minimal vertex cuts separating vertices $u_1$ and $v_1$ and let $G_2$ be a graph on $m+2$ vertices with $g(m)$ minimal vertex cuts separating $u_2$ and $v_2$. Let $G$ be the graph obtained by taking the disjoint union of $G_1$ and $G_2$ and merging the pair of vertices $u_1, u_2$ into a vertex $u$ as well as the pair of vertices $v_1, v_2$ into a vertex $v$. Then, a set $T$ is a minimal set separating $u$ and $v$ in $G$ if and only if for $i \in [2],$ the set $T \cap V(G_i)$ is a minimal set separating $u_i$ and $v_i$ in $G_i$. Therefore, $\mc_{u,v}(G) = g(n) g(m)$ and $G$ clearly has $n+2 + m+2 - 2 = n + m + 2$ vertices.

        By Fekete's lemma, $\lim_n g(n)^{1/n}$ exists, finishing the proof.
    \end{proof}

    By the above discussion, we have
    \begin{equation} \label{eq:equal-limits}
        \alpha = \lim_n c(n)^{1/n} = \lim_n g(n)^{1/n}.
    \end{equation}
    
    \begin{proof}[Proof of Theorem~\ref{thm:main}]
        By \eqref{eq:equal-limits}, it is enough to show that $g(n) \le 2^{(1+o(1))H(1/3) n}.$
        
        Let $G$ be a graph with $n+2$ vertices and let $u, v$ be two fixed vertices of $G$. Let $T$ be an arbitrary inclusion-wise minimal set separating $u$ and $v$ in $G$. Let $S_u$ be the component containing $u$ in $G \setminus T$ and let $S_v$ be the component containing $v$ in $G \setminus T$. Clearly $S_u, S_v, T$ are pairwise disjoint.

        Let $N(X)$ denote the outer neighbourhood\footnote{The outer neighbourhood of $X$ is the set of all vertices in $V(G) \setminus X$ that have a neighbour in $X$.} of $X$ in $G$. 

        \begin{claim}
            $N(S_u) = N(S_v) = T.$
        \end{claim}
        \begin{proof}
            We show that $N(S_u) = T$ and then $N(S_v) = T$ follows analogously. Indeed, suppose first there is a vertex $x \in N(S_u) \setminus T.$ Then there is a vertex $y \in S_u$ such that $yx \in E(G)$, implying that $x \in S_u$, a contradiction. This shows that $T \supseteq N(S_u)$. Now assume that there is a vertex $x \in T \setminus N(S_u).$ Let $T' = T \setminus \{x\}.$ Then, $T'$ also separates $u$ and $v$ in $G$. Indeed, as $x$ has no neighbours in $S_u$, the component of $u$ in the graph $G \setminus T'$ equals $S_u$. In particular, it does not contain $v$, so $T'$ also separates $u$ and $v$, contradicting the minimality of $T$. This shows that $T \subseteq N(S_u)$. We conclude that $N(S_u) = T$.
        \end{proof}
        
        As $S_u, S_v, T$ are disjoint, one of these sets has size at most $m \coloneq \lfloor(n+2) / 3 \rfloor.$ Thus $T$ has size at most $m$ or is equal to $N(S)$ for some set $S$ of size at most $m$.
        
        Let $\binom{V(G)}{\le m}$ denote the set of all subsets of $V(G)$ of size at most $m$. Letting $\calT$ denote the set of all minimal sets separating $u$ and $v$, we have shown that

        \[ \calT \subseteq \binom{V(G)}{\le m} \cup \left\{ N(S) \, \vert \, S \in \binom{V(G)}{\le m} \right\}. \]
        
        This implies that 
        \[ \mc_{u,v}(G) = |\calT| \le 2 \sum_{k=0}^{m} \binom{n+2}{k} \le 2 \cdot 2^{H(m / (n+2))} = 2^{(1+o(1)) H(1/3) n}. \qedhere \]
    \end{proof}    
    
    \textbf{Acknowledgement} I would like to thank Hans Raj Tiwary for informing me that the results of this paper had been known prior to my work and for pointing me to reference~\cite{gaspers-mackenzie}. I am grateful to the anonymous referee for their comments that improved the presentation of the paper as well as Yuval Wigderson for helpful discussions about the problem and for valuable comments on an earlier draft.

    \bibliographystyle{plain}
    \bibliography{references}
\end{document}